\documentclass[12pt,reqno]{amsart}
\usepackage[utf8]{inputenc}
\usepackage[margin=1.2in]{geometry}
\usepackage[lite]{amsrefs}
\usepackage{amsfonts}
\usepackage{mathrsfs}
\usepackage{amscd}
\usepackage{hyperref}
\usepackage{physics}
\usepackage{amsmath}
\usepackage{mathtools}
\usepackage{amssymb}
\usepackage{multicol}
\usepackage{graphicx}
\usepackage{wrapfig}
\usepackage{float}
\usepackage[english]{babel}
\usepackage[utf8]{inputenc}
\usepackage{fancyhdr}
\usepackage{amsthm}
\usepackage[all]{xy}
\usepackage{tikz-cd}
\usepackage{bbold}
\usepackage{tabularx,booktabs,caption,ragged2e}
\newcolumntype{L}{>{\RaggedRight\arraybackslash}X}
\newtheorem*{ThmA}{Theorem A}
\newtheorem*{ThmB}{Theorem B}
\newtheorem*{question}{Question}

\newtheorem{thm}{Theorem}[section]

\newtheorem{lemma}[thm]{Lemma}
\newtheorem{propn}[thm]{Proposition}

\theoremstyle{definition}
\newtheorem{defn}[thm]{Definition}

\theoremstyle{remark}
\newtheorem{remark}[thm]{Remark}

\newcommand{\SC}{\bar{\mathbf{s}}}
\newcommand{\TC}{\bar{\mathbf{t}}}
\newcommand{\MC}{\bar{\mathbf{m}}}
\newcommand{\UC}{\bar{\mathbf{u}}}
\newcommand{\IC}{\bar{\mathbf{\iota}}}

\newcommand{\SG}{{\mathbf{s}}}
\newcommand{\TG}{{\mathbf{t}}}
\newcommand{\MG}{{\mathbf{m}}}
\newcommand{\UG}{{\mathbf{u}}}
\newcommand{\IG}{{\mathbf{\iota}}}

\newcommand{\ST}{\tilde{\mathbf{s}}}
\newcommand{\TT}{\tilde{\mathbf{t}}}
\newcommand{\MT}{\tilde{\mathbf{m}}}
\newcommand{\UT}{\tilde{\mathbf{u}}}
\newcommand{\IT}{\tilde{\mathbf{\iota}}}

\numberwithin{equation}{section}

\usepackage{graphicx} % Required for inserting images

\title{Magnetic cotangent groupoids}

\author{Dadi Ni}%{$^\underline{\times}$}\\
\address{School of Mathematics and Statistics, Henan University, China} %\\\vspace{1mm} \\
\email{\href{mailto:nidd@henu.edu.cn}{nidd@henu.edu.cn}}
\author{Kaichuan Qi}
\address{Department of Mathematics, Penn State University, USA} %\\\vspace{1mm} \\
\email{\href{mailto:kaichuan@psu.edu}{kaichuan@psu.edu}}
\thanks{Ni is  supported by the Natural Science Foundation of Henan Province
 (No. 252300421766). Qi's research is partially supported by the National Science Foundation (award DMS-2302447).}

\begin{document}

\begin{abstract}

We introduce magnetic cotangent groupoids on the cotangent bundle
of a Lie groupoid using normalized cocycles, and
characterize when magnetic symplectic forms are multiplicative
on these groupoids. We show that $\mathbb{S}^1$-central extensions
of Lie groupoids yield such structures via symplectic reduction.
For a Kac--Moody extension, the resulting deformation connects
the coadjoint and gauge action groupoids, which are not
Morita equivalent.

%{\color{purple}
%We study affine deformations of the cotangent groupoid $T^*\mathcal{G}\rightrightarrows A^*$ governed by a one-form $\gamma\in\Omega^1(\mathcal G^{(2)})$, and characterize the conditions on $\gamma$ under which this construction yields a Lie groupoid. We provide an equivalent, more geometric construction: such deformations arise from $\mathbb S^1$-central extensions of Lie groupoids via symplectic reduction at a nonzero moment level, and we identify the reduced symplectic form as a multiplicative magnetic form.  As an
%application, we show that for the Kac--Moody extension the deformation yields nontrivial deformations of the quotient stack $[G/G]$ together with its basic $\mathbb S^1$-gerbe.
%}

 % We study an explicit deformation of the cotangent groupoid by a correction term. We then show that such deformations arise naturally from $\mathbb{S}^1$-central extensions via symplectic reduction. In the case of Kac-Moody extensions, we obtain deformations of stacks and $\mathbb{S}^1$-gerbes.
\end{abstract}


\maketitle

% \tableofcontents
% \parskip = 2em 

\section{Introduction}

It is a classical result by Coste-Dazord-Weinstein \cite{MR0996653} that the cotangent bundle of any Lie groupoid also carries a Lie groupoid structure, called the cotangent groupoid. Further, the canonical symplectic form is compatible with the groupoid structure, this makes the cotangent groupoid a fundamental example of a symplectic groupoid. Cotangent groupoids have had a profound influence on Poisson geometry and other related fields, including the theory of Lie bialgebroids \cite{MR1262213}, VB-groupoids and representations up to homotopy \cite{GSM17}, multiplicative Dirac geometry \cite{O13}, integrable systems \cite{Li03}, Hamilton--Jacobi theory \cite{FMMdDV17}, and microlocal analysis \cite{LV17}.

There is another important class of symplectic forms on the cotangent bundle, namely the magnetic forms, which are obtained by adding the  pullback of a closed two-form from the base to the canonical symplectic form. Magnetic symplectic forms on cotangent bundles arise naturally in the Hamiltonian description of charged particles and in symplectic reduction,  through the minimal-coupling constructions of Sternberg and Weinstein \cite{S77,W78}. They play an important role in the study of magnetic flows and their periodic trajectories \cite{G96}, integrable systems on homogeneous spaces \cite{BJ08}, and magnetic deformation quantization \cite{MP04}. Unlike the canonical symplectic form, a magnetic symplectic form need not be multiplicative with respect to the standard cotangent groupoid structure.

\begin{question}
What is the analogue of cotangent groupoids for magnetic forms? I.e. Is there a Lie groupoid structure on $T^*\mathcal G$ with respect to which
the magnetic symplectic forms
are multiplicative?
\end{question}

One possible answer to the question can be provided by cotangent reduction at a nonzero moment level from central extensions. The case of Kac-Moody extensions is well-developed \cite{MW98} and leads to equivariant gerbes \cite{B-X-Z}. See \cite{H08} for generalizations and applications in geometric quantization. In this setting, the reduced space inherits a groupoid structure, and a choice of connection identifies it with the cotangent bundle \cite{MR0448428,MR1171218,MR0604285}. One can therefore obtain a groupoid structure for $T^*\mathcal{G}$ compatible with the magnetic form.

The central-extension approach requires additional geometric data. We instead construct the groupoid structure directly on \(T^*\mathcal G\), using a one-form \(\gamma\in\Omega^1(\mathcal G^{(2)})\) to introduce affine corrections to the cotangent structure maps. We characterize precisely when these maps define a Lie groupoid and when a magnetic symplectic form is multiplicative. The construction does not assume the existence of a central extension.

\begin{ThmA}\label{thm:groupoid-cond}
Let $\gamma \in \Omega^1(\mathcal{G}^{(2)})$. Consider the maps
$$
\hat{\SG},\hat{\TG},\hat{\MG},\hat{\UG},\hat{\IG}
$$
defined in \eqref{eq:groupoid-def-1}-\eqref{eq:inverse-map}, obtained from
the usual cotangent groupoid structure maps by adding affine correction terms
determined by $\gamma$. These maps make
$T^*\mathcal{G}\rightrightarrows A^*$ into a Lie groupoid if and only if the
following conditions hold:
\begin{itemize}
\item[(1)] Cocycle condition: $\partial \gamma = 0$;
\item[(2)] Degeneracy normalization: $(\sigma_0)^*\gamma = 0$ and
$(\sigma_1)^*\gamma = 0$.
\end{itemize}
When the above conditions hold, we call the resulting groupoid the
\textbf{affine cotangent groupoid} associated with $\gamma$, and denote it by
$(T^*\mathcal G)_\gamma$.

For a closed two-form $\omega_B\in\Omega^2(\mathcal G)$, the magnetic form
\[
\omega_M=\omega_{\mathrm{can}}+({\rm pr}_{\mathcal G})^*\omega_B
\] is multiplicative on $(T^*\mathcal G)_\gamma$
if and only if
$$
\partial\omega_B=d\gamma .
$$
In this case, we call the resulting symplectic groupoid a \textbf{magnetic cotangent groupoid} associated with $(\gamma, \omega_B)$.
\end{ThmA}

Scaling \((\gamma,\omega_B)\) by \(r\in\mathbb R\) gives an affine deformation of the standard cotangent groupoid. This family determines an infinitesimal deformation class \cite{MR5019585}. 

To explain this construction geometrically, we focus on
\(\mathbb S^1\)-central extensions of Lie groupoids. Let
\[
M\times \mathbb S^1 \longrightarrow \widetilde{\mathcal G}
\stackrel{p}{\longrightarrow} \mathcal G
\]
be an \(\mathbb S^1\)-central extension of Lie groupoids, viewed as a
principal \(\mathbb S^1\)-bundle over \(\mathcal G\). For a normalized
connection \(\theta\in\Omega^1(\widetilde{\mathcal G})\), the forms
\(\partial\theta\) and \(d\theta\) descend to forms
\[
\gamma\in\Omega^1(\mathcal G^{(2)}),
\qquad
\omega_B\in\Omega^2(\mathcal G),
\]
and satisfy
\(
\partial\omega_B=d\gamma
\)
\cite{stacks_B-X}. Here $\omega_B$ is the curvature form, while $\gamma$
measures the failure of $\theta$ to be multiplicative. The next theorem
identifies the corresponding magnetic cotangent groupoid with the symplectic groupoid
obtained by reducing \(T^*\widetilde{\mathcal G}\) at moment level \(1\).

\begin{ThmB}
Let \(\widetilde{\mathcal G}\to\mathcal G\) be an
\(\mathbb S^1\)-central extension of Lie groupoids. Let \(
\Gamma=\mu^{-1}(1)/\mathbb S^1 
\) be the reduced symplectic groupoid obtained from symplectic reduction of
\((T^*\widetilde{\mathcal G},\omega_{\mathrm{can}})\) at moment level \(1\). Let
\(\theta\in\Omega^1(\widetilde{\mathcal G})\) be a normalized principal
connection, with associated $(\gamma,\omega_B)$ as above. 

Then there exists  a natural diffeomorphism
\(
\Gamma \cong T^*\mathcal G
\),
under which the reduced groupoid structure is identified with the
affine cotangent groupoid
\((T^*\mathcal G)_{\gamma}\),  and the
reduced symplectic form is 
\(
\omega_{\mathrm{can}}+({\rm pr}_{\mathcal G})^*\omega_B .
\)
\end{ThmB}

Different normalized connections give gauge-equivalent pairs
$(\gamma,\omega_B)$. Note that the fact $\Gamma$ is a symplectic groupoid is recorded by Hawkins \cite{H08}, but a detailed proof is missing. Our approach also provides a proof for this result.

In Section \ref{sec:Kac-Moody-section}, we show that the deformation arising
from the Kac-Moody extension is geometrically nontrivial. As the reduction
level varies from $0$ to $1$, the resulting Lie groupoids interpolate between
the coadjoint action groupoid and the gauge action groupoid, which are not
Morita equivalent. At the level of \(\mathbb S^1\)-gerbes, presented
by the central extensions of Lie groupoids
\[
\mu^{-1}(r)\longrightarrow \mu^{-1}(r)/\mathbb S^1, \quad r\in\mathbb{R},
\]
the level-$1$
gerbe recovers the basic gerbe over \([G/G]\)
\cite{B-X-Z}. When \(G\) is compact, simple, and simply
connected, with the basic normalization of the Kac-Moody extension, the
Dixmier--Douady class of the latter represents the generator of
\(H^3_G(G;\mathbb Z)\). The deformation viewpoint thus places the basic
gerbe construction in a family determined by the extension, whose endpoints
lie over different quotient stacks.

Finally, the magnetic cotangent groupoids arising from central extensions encode the affine Poisson structures on the base \(A^*\).
\vspace{1em}

\noindent\textbf{Conventions.}
We follow the groupoid convention of \cite{mackenzie2005general}.
For $\mathcal G\rightrightarrows M$, the structure maps
$\SG,\TG,\MG,\UG,\IG$ denote source, target, multiplication,
unit and inverse, respectively. We write $gh=\MG(g,h)$ when
$\SG(g)=\TG(h)$, $1_x=\UG(x)$, and $g^{-1}=\IG(g)$.
Let $\mathcal G^{(p)}$ be the space of composable $p$-tuples.
For $\alpha\in\Omega^\bullet(\mathcal G)$ and
$\gamma\in\Omega^\bullet(\mathcal G^{(2)})$, the simplicial differential is
\[
\begin{aligned}
\partial\alpha
  &= \operatorname{pr}_1^*\alpha+\operatorname{pr}_2^*\alpha
     -\MG^*\alpha,\\
\partial\gamma
  &= \operatorname{pr}_{23}^*\gamma-(\MG\times\mathrm{id})^*\gamma\\
  &\quad +(\mathrm{id}\times\MG)^*\gamma
     -\operatorname{pr}_{12}^*\gamma,
\end{aligned}
\]
where $\operatorname{pr}_{12},\operatorname{pr}_{23}:\mathcal G^{(3)}
\to\mathcal G^{(2)}$ are the evident projections. 
The degeneracy maps are $\sigma_0(g)=(1_{\TG(g)},g)$ and
$\sigma_1(g)=(g,1_{\SG(g)})$. The Lie algebroid of $\mathcal{G}$ is denoted as $A = \ker(\SG_*)|_M$.
For $\tilde{\mathcal G}$, $T^*\tilde{\mathcal G}$, the standard
$T^*\mathcal G$, and the reduced groupoid $\Gamma$, we use the
structure-map tuples $(\ST,\TT,\MT,\UT,\IT)$,
$(\SC,\TC,\MC,\UC,\IC)$,
$(\SG_0,\TG_0,\MG_0,\UG_0,\IG_0)$, and
$(\hat\SG,\hat\TG,\hat\MG,\hat\UG,\hat\IG)$, respectively.

\section{Deforming the cotangent groupoid}\label{sec:deform-grpd}

\subsection{Affine cotangent groupoids}
Let $\SG,\TG,\MG,\UG,\IG$ denote the structure maps of $\mathcal{G}\rightrightarrows M$.
Fix a 1-form $\gamma \in \Omega^1(\mathcal{G}^{(2)})$. We then construct a deformation of the cotangent groupoid $T^*\mathcal{G}$ governed by $\gamma$.

\begin{enumerate}
\item The source map $\hat{\SG}\colon T^*\mathcal{G}\to A^*$ is defined by
    \begin{equation}\label{eq:groupoid-def-1}
\langle \hat{\SG}(\xi), a\rangle:=\langle \xi,(L_k)_*(a-\UG_*\TG_*a)\rangle+\gamma(0_k,\IG_*a),  \end{equation}
where $k \in \mathcal{G}$,  $\xi \in T^*_k\mathcal{G}$,  $a\in A_{\SG(k)}$ and $(0_k,\IG_*a) \in T_{(k,\UG\SG(k))}\mathcal{G}^{(2)}$.
\item The target map $\hat{\TG}\colon T^*\mathcal{G}\to A^*$ is defined by
\begin{equation}\label{eq:groupoid-def-2}
\langle \hat{\TG}(\xi), b\rangle :=\langle \xi,(R_k)_*b\rangle-\gamma(b,0_k),   
\end{equation}
where $k \in \mathcal{G}$,  $\xi \in T^*_k\mathcal{G}$,  $b\in A_{\TG(k)}$ and $(b,0_k) \in T_{(\UG\TG(k),k)}\mathcal{G}^{(2)}$.
\item The multiplication map $\hat{\MG}\colon (T^*\mathcal{G})^{(2)}\to T^*\mathcal{G}$ is determined by
\begin{equation}\label{eq:groupoid-def-3}
\langle \hat{\MG}(\xi_1,\xi_2),\MG_*(w_1,w_2)\rangle := \langle\xi_1,w_1\rangle+ \langle \xi_2,w_2\rangle+\gamma(w_1,w_2),    
\end{equation}
where $k_1, k_2 \in \mathcal{G}$ with $\SG(k_1) = \TG(k_2)$, and $\xi_1\in T^*_{k_1}\mathcal{G}, \xi_2\in T^*_{k_2}\mathcal{G}$ with $\hat{\SG}(\xi_1)=\hat{\TG}(\xi_2)$, and $w_1\in T_{k_1}\mathcal{G}, w_2\in T_{k_2}\mathcal{G}$ with ${\SG}_*(w_1)={\TG}_*(w_2)$.
\item The unit map $\hat{\UG}\colon A^*\to T^*\mathcal{G}$ is defined as 
\begin{equation}\label{eq:unit-map}
    \langle \hat{\UG}(\alpha),v\rangle:=\langle \alpha,v-\UG_*\SG_*v\rangle,
\end{equation}
where $x\in M$, $\alpha\in A_x^*$ and $v\in T_{\UG(x)}\mathcal{G}$.
\item The inverse map $\hat{\iota}\colon T^*\mathcal{G}\to T^*\mathcal{G}$ is defined by
\begin{equation}\label{eq:inverse-map}
\langle \hat{\iota}(\xi), w\rangle:=-\langle\xi,\iota_*w\rangle - \gamma(\iota_*w,w),  
\end{equation}
where $k \in \mathcal{G}$,  $\xi \in T^*_k\mathcal{G}$ and $w\in T_{\iota(k)}\mathcal{G}$.
\end{enumerate}

Here, $L, R$ denote left and right multiplication on the corresponding fibers. When $\gamma=0$, we recover the standard cotangent groupoid.

%Here and below, expressions such as $\gamma(0_k,\iota_*a)$ are evaluated at the corresponding point of $\mathcal{G}^{(2)}$. For instance, $(0_k,\IG_*a) \in T_{(k,\UG(\SG(k)))}\mathcal{G}^{(2)}$.

\begin{defn}
We call $T^*\mathcal{G}$  an \textbf{affine cotangent groupoid}, if $T^*\mathcal{G}$ is a Lie groupoid over $A^*$ with structure maps $(\hat{\SG},\hat{\TG},\hat{\MG}, \hat{\UG},\hat{\iota})$ as defined above.
\end{defn}

\subsection{Proof of Theorem A}

We divide the proof of Theorem A into conditions for different classes of groupoid axioms. We first derive the condition for $T^*\mathcal{G}$ being a semigroupoid. Here, a semigroupoid means that all the groupoid identities involving $\hat{\SG},\hat{\TG},\hat{\MG}$ hold, while the existence of the unit map and the inverse map is not required. We need the following lemma.

\begin{lemma}\label{prop:elementinkernel}
For any \(k_1, k_2 \in \mathcal{G}\) satisfying \(\SG(k_1) = \TG(k_2)\), and any tangent vectors \(w_1\in T_{k_1}\mathcal{G}\), \(w_2\in T_{k_2}\mathcal{G}\) with \(\SG_*(w_1)=\TG_*(w_2)\), then \(\MG_*(w_1,w_2) = 0\), if and only if there exists some \(a\in A_{\SG(k_1)}\) such that
\[
w_1=(L_{k_1})_*\iota_*a \quad \text{and} \quad w_2=(R_{k_2})_*a,
\]
where $L_{k_1}$ is the left multiplication by $k_1$ and $R_{k_2}$ is the right multiplication by $k_2.$
\end{lemma}

\begin{proof}
Since $\SG_*w_2 = \SG_*\MG_*(w_1,w_2)$, we have $\SG_*w_2 = 0$. 
Hence $(R_{\iota(k_2)})_*w_2 \in A_{\SG(k_1)}$, and there exists 
$a \in A_{\SG(k_1)}$ such that $w_2 = (R_{k_2})_*a$.

Similarly, from $\TG_*w_1 = \TG_*\MG_*(w_1,w_2)$, we obtain 
$\SG_*\iota_*w_1 = 0$. It follows that 
$(R_{k_1})_*\iota_*w_1 \in A_{\SG(k_1)}$, so there exists 
$b \in A_{\SG(k_1)}$ such that $w_1 = (L_{k_1})_*\iota_*b$.

On the other hand, the condition $\SG_*(w_1) = \TG_*(w_2)$ implies 
$\TG_*a = \TG_*b$. Combining this with the assumption
\[
\MG_*\big((L_{k_1})_*\iota_*b,(R_{k_2})_*a\big) = 0,
\]
we conclude that $a = b$. The converse follows by a similar argument.
\end{proof}

\begin{propn}\label{propn:semi-groupoid}
Endowed with the structure maps $\hat{\SG}$, $\hat{\TG}$, $\hat{\MG}$ defined above, 
$T^*\mathcal{G}$ is a semigroupoid if and only if
\begin{equation}\label{eq:semigroupoid-cond}
    \partial \gamma = 0, 
    \quad\text{and}\quad 
    \gamma(\iota_*a,a)=0, \quad\text{for all } a\in A.
\end{equation}
\end{propn}

\begin{proof}
We first show that \eqref{eq:semigroupoid-cond} makes $\hat{\MG}$ well-defined. 
By Lemma~\ref{prop:elementinkernel}, it suffices to verify that the pairing
\begin{equation}\label{eq:pairing-kernel}
\begin{aligned}
&\big\langle \hat{\MG}(\xi_1,\xi_2),\, 
\MG_*\big((L_{k_1})_*\iota_*a,(R_{k_2})_*a\big) \big\rangle \\
&\qquad= \gamma\big((L_{k_1})_*\iota_*a,(R_{k_2})_*a\big) 
+ \langle \xi_1,(L_{k_1})_*\iota_*a\rangle 
+ \langle (R_{k_2})_*a,\xi_2\rangle
\end{aligned}
\end{equation}
vanishes for all $a\in A_{\SG(k_1)}$.

Using Equations \eqref{eq:groupoid-def-1}, and 
$\iota_*a = -a + \UG_*\TG_*a$, we compute
\begin{align*}
\langle \hat{\SG}(\xi_1), a\rangle 
&= \langle {\SG}_0(\xi_1), a\rangle + \gamma(0_k,\IG_*a) \\
&= \langle \xi_1, -(L_{k_1})_*\iota_*a\rangle + \gamma(0_k,\IG_*a).
\end{align*}
Similarly, by Equations \eqref{eq:groupoid-def-2},
\[
\langle \hat{\TG}(\xi_2), a\rangle 
= \langle \xi_2, (R_{k_2})_*a\rangle - \gamma(a,0_{k_2}).
\]
Since $\hat{\SG}(\xi_1) = \hat{\TG}(\xi_2)$, we obtain
\begin{equation}\label{eq:xi-relation}
\langle \xi_1,(L_{k_1})_*\iota_*a\rangle 
+ \langle (R_{k_2})_*a,\xi_2\rangle 
= \gamma(0_{k_1},\iota_*a) + \gamma(a,0_{k_2}).
\end{equation}
Substituting \eqref{eq:xi-relation} into \eqref{eq:pairing-kernel} yields
\begin{equation}\label{eq:pairing-simplified}
\begin{aligned}
&\big\langle \hat{\MG}(\xi_1,\xi_2),\, 
\MG_*\big((L_{k_1})_*\iota_*a,(R_{k_2})_*a\big) \big\rangle \\
&\qquad= \gamma\big((L_{k_1})_*\iota_*a,(R_{k_2})_*a\big) 
+ \gamma(0_{k_1},\iota_*a) + \gamma(a,0_{k_2}).
\end{aligned}
\end{equation}

Now apply the cocycle condition $\partial\gamma=0$ together with 
Lemma~\ref{prop:elementinkernel}. We have
\begin{equation}\label{eq:cocycle-1}
\begin{aligned}
0 &= (\partial\gamma)(0_{k_1},\iota_*a,(R_{k_2})_*a) \\
  &= \gamma(\iota_*a,(R_{k_2})_*a) 
   - \gamma\big((L_{k_1})_*\iota_*a,(R_{k_2})_*a\big) 
   - \gamma(0_{k_1},\iota_*a),
\end{aligned}
\end{equation}
and
\begin{equation}\label{eq:cocycle-2}
\begin{aligned}
0 &= (\partial\gamma)(\iota_*a,a,0_{k_2}) \\
  &= \gamma(a,0_{k_2}) + \gamma(\iota_*a,(R_{k_2})_*a) - \gamma(\iota_*a,a).
\end{aligned}
\end{equation}
Combining Equations \eqref{eq:cocycle-1}, \eqref{eq:cocycle-2} with \eqref{eq:pairing-simplified}, 
we arrive at
\[
\big\langle \hat{\MG}(\xi_1,\xi_2),\, 
\MG_*\big((L_{k_1})_*\iota_*a,(R_{k_2})_*a\big) \big\rangle 
= \gamma(\iota_*a,a).
\]
Hence, if $\gamma(\iota_*a,a)=0$ for all $a\in A_{\SG(k_1)}$, then $\hat{\MG}$ is well-defined. Using $\partial \gamma=0$ again, it follows that $\hat{\MG}$ is associative.

It remains to verify the semigroupoid axioms. 
Under Equation \eqref{eq:semigroupoid-cond}, we check that 
$\hat{\SG}(\hat{\MG}(\xi_1,\xi_2))=\hat{\SG}(\xi_2)$ and 
$\hat{\TG}(\hat{\MG}(\xi_1,\xi_2))=\hat{\TG}(\xi_1)$. 
For the first identity, a direct computation gives
\begin{equation*}
\begin{aligned}
\langle \hat{\SG}(\hat{\MG}(\xi_1,\xi_2)), a\rangle 
&= \langle \hat{\SG}(\xi_2),a\rangle 
+ \gamma(0_{k_1k_2},\iota_*a) 
- \gamma(0_{k_1},(L_{k_2})_*\iota_*a) 
- \gamma(0_{k_2},\iota_*a).
\end{aligned}
\end{equation*}
The last three terms vanish because $(\partial\gamma)(0_{k_1},0_{k_2},\iota_*a)=0$. 
The second identity follows analogously from 
$(\partial\gamma)(b,0_{k_1},0_{k_2})=0$ for $b\in A_{\TG(k_1)}$.

Thus Equation \eqref{eq:semigroupoid-cond} implies that $T^*\mathcal{G}$ is a semigroupoid. 
Conversely, if $T^*\mathcal{G}$ carries a semigroupoid structure, then associativity of $\hat{\MG}$ forces the cocycle condition 
$\partial\gamma=0$.
Furthermore, the well-definedness of $\hat{\MG}$ yields $\gamma(\iota_*a,a)=0$ for all $a\in A_{\SG(k_1)}$.
\end{proof}

For the unit and inverse axioms to hold, we need additional conditions.

\begin{propn}\label{propn:unit-cond}
Assume Condition \eqref{eq:semigroupoid-cond} holds. Then the unit axioms for 
$\hat{\UG}$ are equivalent to
\begin{equation}\label{eq:unit-cond}
(\sigma_0)^*\gamma = 0, \qquad (\sigma_1)^*\gamma = 0.
\end{equation}
\end{propn}

\begin{proof}
We first show that $\hat{\TG}\circ\hat{\UG} = \mathrm{id}_{A^*}$ if and only if 
$\gamma(a,0_{\UG(x)})=0$ for all $a\in A_x$. 
By Equations \eqref{eq:groupoid-def-2} and \eqref{eq:unit-map}, for any $\alpha\in A^*_x$ we have
\[
\langle \hat{\TG}(\hat{\UG}(\alpha)),a\rangle 
= \langle {\TG}_0({\UG}_0(\alpha)),a\rangle - \gamma(a,0_{\UG(x)})
= \langle\alpha,a\rangle - \gamma(a,0_{\UG(x)}).
\]
Hence $\hat{\TG}\circ\hat{\UG} = \mathrm{id}_{A^*}$ precisely when $\gamma(a,0_{\UG(x)})=0$ 
for all $a\in A_x$.

Next, we show that $$\hat{\MG}\big(\xi,\hat{\UG}(\hat{\SG}(\xi))\big)=\xi ,\quad \text{ for all}~\xi\in T^*_k\mathcal{G},$$ if and only if $(\sigma_1)^*\gamma=0$. 
For any $w\in T_k\mathcal{G}$, note that $\MG_*(w,\UG_*\SG_*w)=w$. 
Using Equation \eqref{eq:unit-map}, we compute
\begin{equation*}
\begin{aligned}
 \langle \hat{\MG}(\xi,\hat{\UG}(\hat{\SG}(\xi))),w\rangle &=  \langle \hat{\MG}(\xi,\hat{\UG}(\hat{\SG}(\xi))),\MG_*(w,\UG_*\SG_*w)\rangle \\
 &= \langle\xi,w\rangle+\langle \UG_0(\hat{\SG}(\xi)),\UG_*\SG_*w\rangle + \gamma(w,\UG_*\SG_*w)\\
 &=\langle\xi,w\rangle+\langle \hat{\SG}(\xi),\UG_*\SG_*w-\UG_*\SG_*\UG_*\SG_*w\rangle + \gamma(w,\UG_*\SG_*w)\\
 &=\langle\xi,w\rangle+ \gamma(w,\UG_*\SG_*w).
\end{aligned}
\end{equation*}
Thus $\hat{\MG}(\xi,\hat{\UG}(\hat{\SG}(\xi)))=\xi$ for all $\xi$ if and only if 
$\gamma(w,\UG_*\SG_*w)=0$ for all $w$, which is precisely $(\sigma_1)^*\gamma=0$.

Note that under $(\sigma_1)^*\gamma=0$, the condition $\gamma(a,0_{\UG(x)})=0$ 
follows automatically.

By the same argument, $\hat{\SG}\circ\hat{\UG} = \mathrm{id}_{A^*}$ if and only if 
$\gamma(0_{\UG(x)},\iota_*a)=0$ for all $a\in A_x$, and 
$\hat{\MG}(\hat{\UG}(\hat{\TG}(\xi)),\xi)=\xi$ for all $\xi\in T^*_k\mathcal{G}$ 
if and only if $(\sigma_0)^*\gamma=0$. 
Moreover, $(\sigma_0)^*\gamma=0$ implies $\gamma(0_{\UG(x)},\iota_*a)=0$.

Therefore, the unit axioms hold if and only if both 
$(\sigma_0)^*\gamma=0$ and $(\sigma_1)^*\gamma=0$.
\end{proof}

Now we make the following convenient observation, which shows that the second equation in Condition \eqref{eq:semigroupoid-cond} 
becomes redundant once Condition \eqref{eq:unit-cond} is imposed.
\begin{lemma}\label{lem:inverse-redundant}
Assume $(\sigma_0)^*\gamma = 0$ and $(\sigma_1)^*\gamma = 0$, then $\gamma(\iota_*a,a) = 0$ for all $a\in A$.
\end{lemma}

\begin{proof}
Since $(\sigma_0)^*\gamma = 0$ and $(\sigma_1)^*\gamma = 0$, we have 
$\gamma(a,0) = 0$ and $\gamma(\UG_*\TG_*a,a) = 0$, for all $a\in A$. 
Using $\iota_*a = -a + \UG_*\TG_*a$, we compute
\[
\gamma(\iota_*a,a) 
= -\gamma(a,0) + \gamma(\UG_*\TG_*a,a) = 0.
\]
\end{proof}

For the inverse axioms, we make use of the following lemma.

\begin{lemma}
Assume that $\sigma^*_0\gamma=0$, $\sigma^*_1\gamma=0$, and $\partial \gamma=0$. Then the following inverse compatibility condition holds:
\begin{equation*}
\gamma(\iota_*v,v) = \gamma(v,\iota_*v), \quad \forall\, v\in T\mathcal{G}.
\end{equation*}
\end{lemma}

\begin{proof}
Since $T\mathcal{G}$ is a Lie groupoid, the groupoid multiplication satisfies the identities
\[
\MG_*(v,\iota_*v)=\UG_*\TG_*v,\qquad 
\MG_*(\iota_*v,v)=\UG_*\SG_*v,\quad \forall\, v\in T_g\mathcal{G}.
\]
Evaluate the simplicial cocycle condition $\partial\gamma=0$ on the composable triple $(g,\iota(g),g)$ and the tangent triple $(v,\iota_*v,v)$. This gives
\[
0=\partial \gamma(v,\iota_*v,v)=\gamma(\iota_*v,v)-\gamma\big(\MG_*(v,\iota_*v),v\big)+\gamma\big(v,\MG_*(\iota_*v,v)\big)-\gamma(v,\iota_*v).
\]
By the pullback vanishing assumptions $\sigma_0^*\gamma=0$ and $\sigma_1^*\gamma=0$, we further obtain
\[
\gamma\big(\MG_*(v,\iota_*v),v\big)=\gamma(\UG_*\TG_*v,v)=\sigma_0^*\gamma(v)=0,
\]
\[
\gamma\big(v,\MG_*(\iota_*v,v)\big)=\gamma(v,\UG_*\SG_*v)=\sigma_1^*\gamma(v)=0.
\]
Substituting these vanishing relations into the cocycle identity yields
\[
\gamma(\iota_*v,v) = \gamma(v,\iota_*v),
\]
which completes the proof.
\end{proof}

\begin{propn}\label{prop: inverse-cond}
Assume $(\sigma_0)^*\gamma=0$, $(\sigma_1)^*\gamma=0$ and $\partial \gamma = 0$. Then the inverse axioms for 
$\hat{\iota}$ hold.
\end{propn}

\begin{proof}
We make use of the inverse compatibility condition in the previous lemma and verify that all inverse axioms hold.

To prove $\hat{\SG}\circ\hat{\iota}=\hat{\TG}$, for any  $\xi\in T^*_k\mathcal{G}$ and 
$b\in A_{\TG(k)}$, we compute
\begin{equation*}
\begin{aligned}
\langle \hat{\SG}(\hat{\iota}(\xi)),b\rangle 
&= \langle \SG_0(\iota_0(\xi)),b\rangle 
   + \gamma((R_k)_*b,(L_{k^{-1}})_*\iota_*b) + \gamma(0_{k^{-1}},\iota_*b) \\
&= \langle \TG_0(\xi),b\rangle - \gamma(b,0_k) 
   + \gamma(b,0_k) + \gamma((R_k)_*b,(L_{k^{-1}})_*\iota_*b) + \gamma(0_{k^{-1}},\iota_*b) \\
&= \langle \hat{\TG}(\xi),b\rangle 
   + \gamma(b,0_k) + \gamma((R_k)_*b,(L_{k^{-1}})_*\iota_*b) + \gamma(0_{k^{-1}},\iota_*b).
\end{aligned}
\end{equation*}
To simplify the last three terms, we use $\gamma(\iota_*v,v)=\gamma(v,\iota_*v)$, which gives
\[
\gamma((R_k)_*b,(L_{k^{-1}})_*\iota_*b) 
= \gamma((L_{k^{-1}})_*\iota_*b,(R_k)_*b).
\]
Here we use the fact $\iota_*(R_k)_*b=(L_{k^{-1}})_*\iota_*b.$
By Equations \eqref{eq:cocycle-1} and \eqref{eq:cocycle-2} in the proof of 
Proposition~\ref{propn:semi-groupoid}, the sum of the last three terms vanishes:
\begin{equation*}
\begin{aligned}
&\gamma((R_k)_*b,(L_{k^{-1}})_*\iota_*b) + \gamma(b,0_k) + \gamma(0_{k^{-1}},\iota_*b) \\
=\;&\gamma((L_{k^{-1}})_*\iota_*b,(R_k)_*b) + \gamma(b,0_k) + \gamma(0_{k^{-1}},\iota_*b) \\
=\;&\gamma(\iota_*b,b) = 0,
\end{aligned}
\end{equation*}
where the last equality follows from Condition \eqref{eq:semigroupoid-cond}. 
Thus $\hat{\SG}\circ\hat{\iota}=\hat{\TG}$. One checks similarly that 
$\hat{\TG}\circ\hat{\iota}=\hat{\SG}$.

It remains to show that $\hat{\MG}(\hat{\iota}(\xi),\xi) = \hat{\UG}(\hat{\SG}(\xi))$. 
Reversing the computation in the first part, we obtain
\[
\langle \hat{\MG}(\hat{\iota}(\xi),\xi), \UG_*\SG_*v\rangle 
= \gamma(\iota_*v,v) - \gamma(v,\iota_*v) = 0.
\]
Hence $\hat{\MG}(\hat{\iota}(\xi),\xi)$ annihilates $\UG_*(TM)$ and therefore lies in 
$\hat{\UG}(A^*) = \UG_0(A^*)$. Since $\hat{\SG}\colon \hat{\UG}(A^*)\to A^*$ is a bijection and
\[
\hat{\SG}\bigl(\hat{\MG}(\hat{\iota}(\xi),\xi)\bigr) 
= \hat{\SG}(\xi),
\]
we conclude that $\hat{\MG}(\hat{\iota}(\xi),\xi) = \hat{\UG}(\hat{\SG}(\xi))$. 
The identity $\hat{\MG}(\xi,\hat{\iota}(\xi)) = \hat{\UG}(\hat{\TG}(\xi))$ follows by an analogous argument.
\end{proof}

\begin{proof}[Proof of Theorem A]
If $(T^*\mathcal{G})_{\gamma}$ is a Lie groupoid, then by Proposition \ref{propn:semi-groupoid} and \ref{propn:unit-cond}, we have that $(\sigma_0)^*\gamma=0$, $(\sigma_1)^*\gamma=0$ and $\partial \gamma = 0$.  Conversely, if $\gamma$ satisfies $(\sigma_0)^*\gamma=0$, $(\sigma_1)^*\gamma=0$ and $\partial \gamma = 0$, then by Lemma \ref{lem:inverse-redundant} and Proposition \ref{propn:semi-groupoid}, we have that $(T^*\mathcal{G})_{\gamma}$ is a semigroupoid with well-defined multiplication map. Further, by  Proposition \ref{propn:unit-cond} and \ref{prop: inverse-cond}, we have the unit and inverse axioms. Thus $(T^*\mathcal{G})_{\gamma}$ is a groupoid. It is clear that the structure maps $\hat{\SG},\hat{\TG},\hat{\MG},\hat{\UG},\hat{\IG}$ are smooth maps and that $\hat{\SG},\hat{\TG}$ are submersions, making $(T^*\mathcal{G})_{\gamma}$ a Lie groupoid.

It remains to check compatibility for magnetic forms. Let $\vartheta_L \in \Omega^1(T^*\mathcal{G})$ denote the Liouville $1$-form on $T^*\mathcal{G}$. For any composable pair $(\xi_1,\xi_2)$ and tangent vectors $(V_1,V_2)$ whose base projections form $w_1,w_2 \in T\mathcal{G}^{(2)}$, a direct computation yields
\begin{equation*}
\begin{aligned}
({\rm pr}_1^*\vartheta_L+{\rm pr}_2^*\vartheta_L-\hat{\MG}^*\vartheta_L)(V_1,V_2)
&= \langle \xi_1,w_1\rangle+\langle \xi_2,w_2\rangle-\langle \hat{\MG}(\xi_1,\xi_2),\MG_*(w_1,w_2)\rangle  \\
&=-\gamma(w_1,w_2)\\
&=-({\rm pr}^*_{\mathcal{G}}\gamma)(V_1,V_2).
\end{aligned}
\end{equation*}
Since $\omega_{\rm can} = d\vartheta_L$, we obtain
\[
\partial \omega_{\rm can}=(d\partial) \vartheta_L = -d({\rm pr}_{\mathcal{G}}^*\gamma)=-{\rm pr}_{\mathcal{G}}^*(d\gamma),
\]
where $\partial$ denotes the simplicial groupoid differential operator associated to $(T^*\mathcal{G})_{\gamma}$. Let ${\rm pr}_{\mathcal{G}}\colon (T^*\mathcal{G})_{\gamma}\rightarrow \mathcal{G}$ be the bundle projection, which induces a projection $q^{(2)}:(T^*\mathcal{G})_{\gamma}^{(2)}\rightarrow \mathcal{G}^{(2)}$. As the projection ${\rm pr}_{\mathcal{G}}\colon (T^*\mathcal{G})_{\gamma}\rightarrow \mathcal{G}$ is a Lie groupoid morphism, the differential operators commute: $\partial \circ {\rm pr}_{\mathcal{G}}^* ={(q^{(2)})}^* \circ \partial$. Consequently, the condition $\partial \omega_M = 0$ is equivalent to
\begin{equation*}
\partial ({\rm pr}_{\mathcal{G}}^* \omega_B) ={(q^{(2)})}^*\partial \omega_B= {(q^{(2)})}^*d\gamma.
\end{equation*}
Since $q^{(2)}$ is a surjective submersion, we conclude that the compatibility of $\omega_M$ is equivalent to $\partial \omega_B = d\gamma$.
\end{proof}

\subsection{Gauge transformations}
As in \cite{MR1973074}, we consider equivalences given by translations in the cotangent fibres.
\begin{defn}
For affine cotangent groupoids $(T^*\mathcal{G})_{\gamma_i}$, $i=1,2$, a
\emph{gauge transformation} is a Lie groupoid isomorphism of the form
$F_\alpha(\xi_g)=\xi_g+\alpha_g$, where $\alpha\in\Omega^1(\mathcal{G})$.
\end{defn}

\begin{propn}\label{prop: gauge-1}
The translation $F_\alpha$ is a gauge transformation if and only if
$\gamma_1-\gamma_2=\partial\alpha$.
\end{propn}

\begin{proof}
If $F_\alpha$ preserves multiplication, comparing both sides of
$F_\alpha(\hat{\MG}_1(\xi_g,\xi_h))
=\hat{\MG}_2(F_\alpha\xi_g,F_\alpha\xi_h)$ using
\eqref{eq:groupoid-def-3} yields $\gamma_1-\gamma_2=\partial\alpha$.

Conversely, assume this identity. Pulling back $\partial\alpha=\gamma_1-\gamma_2$ along $\sigma_0$ gives
$\TG^*(\UG^*\alpha)=0$, hence $\UG^*\alpha=0$. Thus $\beta_x(a):=\alpha_{1_x}(a)$, $a\in A_x$, defines a section
$\beta\in\Gamma(A^*)$. The formulas for source and target in
\eqref{eq:groupoid-def-1}--\eqref{eq:groupoid-def-2}, together with
$\gamma_1-\gamma_2=\partial\alpha$, give
\[
\begin{aligned}
\hat{\SG}_2(F_\alpha\xi_g)&=\hat{\SG}_1(\xi_g)+\beta_{\SG(g)},\\
\hat{\TG}_2(F_\alpha\xi_g)&=\hat{\TG}_1(\xi_g)+\beta_{\TG(g)}.
\end{aligned}
\]
Hence $F_\alpha$ preserves composability and units, covering
$\eta_x\mapsto\eta_x+\beta_x$ on $A^*$. Moreover, for composable
$\xi_g,\xi_h$ and $(v,w)\in T_{(g,h)}\mathcal{G}^{(2)}$,
\[
\begin{aligned}
&\big\langle F_\alpha(\hat{\MG}_1(\xi_g,\xi_h))
-\hat{\MG}_2(F_\alpha\xi_g,F_\alpha\xi_h),\,\MG_*(v,w)\big\rangle\\
&\hspace{4em}=(\gamma_1-\gamma_2-\partial\alpha)(v,w)=0.
\end{aligned}
\]
Thus $F_\alpha$ preserves multiplication. Being a diffeomorphism,
it is a Lie groupoid isomorphism.
\end{proof}

Let $\omega_{M_i}=\omega_{\rm can}+{\rm pr}_{\mathcal{G}}^*\omega_{B_i}$.
Since $F_\alpha^*\vartheta_L=\vartheta_L+{\rm pr}_{\mathcal{G}}^*\alpha$,
we have
\begin{equation}
F_\alpha^*\omega_{M_2}
=\omega_{\rm can}+{\rm pr}_{\mathcal{G}}^*(\omega_{B_2}+d\alpha).
\end{equation}
A gauge transformation is \emph{magnetic} if it is also a symplectomorphism.

\begin{propn}\label{prop:gauge-2}
The translation $F_\alpha$ is an isomorphism between the magnetic cotangent
groupoids associated with $(\gamma_1,\omega_{B_1})$ and
$(\gamma_2,\omega_{B_2})$ if and only if
\begin{equation}
\gamma_1-\gamma_2=\partial\alpha,
\qquad \omega_{B_1}-\omega_{B_2}=d\alpha.
\end{equation}
\end{propn}

\section{Central extensions and magnetic cotangent groupoids}

\subsection{Reduction and the affine cotangent structure}\label{sec:central-affine}
Let
\begin{equation}\label{eq:central-groupoid}
M\times\mathbb S^1\longrightarrow\widetilde{\mathcal G}
\stackrel{p}{\longrightarrow}\mathcal G
\end{equation}
be an $\mathbb S^1$-central extension of Lie groupoids, with Lie
algebroids $\widetilde A$ and $A$. The central $\mathbb{S}^1$-action satisfies
\begin{equation}\label{Eqt:centrality}
\Phi_r(\tilde g)=i(\TT(\tilde g),r)\tilde g
=\tilde g\,i(\ST(\tilde g),r).
\end{equation}
Write $\varepsilon\in\Gamma(\widetilde A)$ for the infinitesimal
central generator. The cotangent lift is Hamiltonian, with moment map
\[
\mu(\xi_{\tilde g})=\langle\xi_{\tilde g},\psi_{\tilde g}\rangle,
\qquad
\psi_{\tilde g}=(L_{\tilde g})_*\varepsilon_{\ST(\tilde g)}
=(R_{\tilde g})_*\varepsilon_{\TT(\tilde g)}.
\]
By the cotangent source and target formulas,
\begin{equation}\label{eq:moment-source-target}
\mu(\xi)=\langle\SC(\xi),\varepsilon\rangle
=\langle\TC(\xi),\varepsilon\rangle.
\end{equation}
Consequently, $K:=\mu^{-1}(1)$ is the restriction of
$T^*\widetilde{\mathcal G}$ to
$N:=\{\lambda\in\widetilde A^*:\langle\lambda,\varepsilon\rangle=1\}$.
Centrality implies that the lifted $\mathbb{S}^1$-action preserves the
cotangent source and target and that
$\MC(\bar\Phi_r\xi,\bar\Phi_s\eta)
=\bar\Phi_{rs}\MC(\xi,\eta)$ for composable covectors.
Thus $\Gamma:=K/\mathbb S^1\rightrightarrows N$ is a Lie groupoid \cite{H08}. 
Symplectic reduction gives it a multiplicative symplectic form
$\omega_\Gamma$.

Choose a principal connection $\theta\in\Omega^1(\widetilde{\mathcal G})$
normalized by $\UT^*\theta=0$. Such a choice always exists: starting
with $\theta_0$, subtract $p^*\lambda$ for a one-form $\lambda$ on
$\mathcal G$ satisfying $\UG^*\lambda=\UT^*\theta_0$.
The resulting splitting of the central Lie algebroid extension identifies
$\widetilde A\simeq A\oplus\mathbb R$ and
$N\simeq A^*$, via $j(\alpha)=(\alpha,1)$.
The forms $\partial\theta$ and $d\theta$ are basic and hence determine
$\gamma\in\Omega^1(\mathcal G^{(2)})$ and
$\omega_B\in\Omega^2(\mathcal G)$ by
\begin{equation}\label{eq:descend-theta}
(p^{(2)})^*\gamma=\partial\theta,\qquad
p^*\omega_B=d\theta.
\end{equation}
In particular, $\partial\gamma=0$,
$\sigma_0^*\gamma=\sigma_1^*\gamma=0$, and
$\partial\omega_B=d\gamma$. Thus $(\gamma,\omega_B)$ satisfies
the conditions of Theorem A.

\begin{thm}\label{thm: red-grpd}
The map
\begin{equation}\label{eq:reduction-beta}
\beta:T^*\mathcal G\longrightarrow\Gamma,\qquad
\beta(\xi_k)=\big[(dp_{\tilde k})^*\xi_k+\theta_{\tilde k}\big],
\quad p(\tilde k)=k,
\end{equation}
is a well-defined isomorphism of Lie groupoids
$(T^*\mathcal G)_\gamma\cong\Gamma$ over $A^*$. Moreover,
\[
\beta^*\omega_\Gamma
=\omega_{\mathrm{can}}+\operatorname{pr}_{\mathcal G}^*\omega_B.
\]
\end{thm}

\begin{proof}
We first recall the cotangent-reduction identification \cite{MR0448428,MR1171218,MR0604285}. For
$k=p(\tilde k)$ and $\xi\in T_k^*\mathcal G$, let
$\tilde\xi=(dp_{\tilde k})^*\xi+\theta_{\tilde k}$. Since
$\theta(\psi)=1$, we have $\mu(\tilde\xi)=1$; the invariance of
$\theta$ makes the class $[\tilde\xi]$ independent of the choice of
$\tilde k$. Conversely, for any $\tilde\xi\in K$, the covector
$\tilde\xi-\theta_{\tilde k}$ annihilates the vertical direction,
so it is uniquely of the form $(dp_{\tilde k})^*\xi$. This proves
that $\beta$ is a diffeomorphism. We identify its base with
$A^*$ using the splitting of $\widetilde A$ determined by $\theta$.

We verify that the transferred structure maps are those of
$(T^*\mathcal G)_\gamma$. Fix $k=p(\tilde k)$ and write
$\tilde\xi=(dp_{\tilde k})^*\xi+\theta_{\tilde k}$. For
$a\in A_{\SG(k)}$ and $b\in A_{\TG(k)}$, let
$\tilde a,\tilde b$ be their $\theta$-horizontal lifts in
$\widetilde A$. Normalization gives
$\theta(\tilde a)=\theta(\tilde b)=0$ and
$\theta(\IT_*\tilde a)=0$, since
$\IT_*\tilde a=-\tilde a+\UT_*\tilde\rho(\tilde a)$
at the unit. The Coste--Dazord--Weinstein formulas and
$(p^{(2)})^*\gamma=\partial\theta$ therefore yield
\[
\begin{aligned}
\langle\SC(\tilde\xi),\tilde a\rangle
 &=\langle\SG_0(\xi),a\rangle
   -\theta_{\tilde k}((L_{\tilde k})_*\IT_*\tilde a)\\
 &=\langle\SG_0(\xi),a\rangle
   +(\partial\theta)(0_{\tilde k},\IT_*\tilde a)\\
 &=\langle\SG_0(\xi),a\rangle+\gamma(0_k,\IG_*a),\\[2pt]
\langle\TC(\tilde\xi),\tilde b\rangle
 &=\langle\TG_0(\xi),b\rangle
   +\theta_{\tilde k}((R_{\tilde k})_*\tilde b)\\
 &=\langle\TG_0(\xi),b\rangle
   -(\partial\theta)(\tilde b,0_{\tilde k})\\
 &=\langle\TG_0(\xi),b\rangle-\gamma(b,0_k).
\end{aligned}
\]
Here $\tilde\rho$ is the anchor of $\widetilde A$.
By \eqref{eq:moment-source-target}, both central components are $1$.
Thus
\[
\SC(\tilde\xi)=(\hat\SG(\xi),1),\qquad
\TC(\tilde\xi)=(\hat\TG(\xi),1),
\]
which are precisely \eqref{eq:groupoid-def-1} and
\eqref{eq:groupoid-def-2}. In particular, $\beta$ preserves the
source and target maps and hence composability.

Next, let $\xi_i\in T_{k_i}^*\mathcal G$ be composable for the affine
structure, and choose composable lifts
$\tilde\xi_i=(dp_{\tilde k_i})^*\xi_i+\theta_{\tilde k_i}$.
For $(v_1,v_2)\in T_{(\tilde k_1,\tilde k_2)}
\widetilde{\mathcal G}^{(2)}$, put $w_i=dp(v_i)$.
The cotangent multiplication rule gives
\[
\begin{aligned}
&\big\langle\MC(\tilde\xi_1,\tilde\xi_2)
   -\theta_{\tilde k_1\tilde k_2},\,
   \MT_*(v_1,v_2)\big\rangle\\
&\quad=\langle\xi_1,w_1\rangle+\langle\xi_2,w_2\rangle
  +\theta_{\tilde k_1}(v_1)+\theta_{\tilde k_2}(v_2)
  -\theta_{\tilde k_1\tilde k_2}(\MT_*(v_1,v_2))\\
&\quad=\langle\xi_1,w_1\rangle+\langle\xi_2,w_2\rangle
  +(\partial\theta)(v_1,v_2)\\
&\quad=\langle\xi_1,w_1\rangle+\langle\xi_2,w_2\rangle
  +\gamma(w_1,w_2).
\end{aligned}
\]
Since $\MC(\tilde\xi_1,\tilde\xi_2)$ has moment value $1$,
the covector on the left is horizontal and descends through $p$.
As $p^{(2)}$ is a surjective submersion and $p$ intertwines
multiplication, the descended covector satisfies exactly
\eqref{eq:groupoid-def-3}. Hence $\beta$ also preserves
multiplication.

For completeness, we check the units and inverses. If
$\alpha\in A_x^*$ and $v\in T_{\UT(x)}\widetilde{\mathcal G}$,
the cotangent unit formula and $\UT^*\theta=0$ give
\[
\begin{aligned}
\langle\UC(\alpha,1),v\rangle
 &=\langle(\alpha,1),v-\UT_*\ST_*v\rangle\\
 &=\langle\UG_0(\alpha),dp(v)\rangle+\theta(v).
\end{aligned}
\]
Thus $\UC(\alpha,1)$ is the $\beta$-lift of $\UG_0(\alpha)$.
Similarly, for $v\in T_{\tilde k^{-1}}\widetilde{\mathcal G}$
and $w=dp(v)$, the cotangent inverse formula gives
\[
\begin{aligned}
&\langle\IC(\tilde\xi)-\theta_{\tilde k^{-1}},v\rangle\\
&\quad=\langle\IG_0(\xi),w\rangle
  -\theta_{\tilde k}(\IT_*v)-\theta_{\tilde k^{-1}}(v)\\
&\quad=\langle\IG_0(\xi),w\rangle
  -(\partial\theta)(\IT_*v,v)\\
&\quad=\langle\IG_0(\xi),w\rangle-\gamma(\IG_*w,w).
\end{aligned}
\]
The second equality uses $\UT^*\theta=0$, since
$\MT_*(\IT_*v,v)$ is tangent to the unit section.
This recovers \eqref{eq:unit-map} and \eqref{eq:inverse-map}.
Therefore $\beta$ is an isomorphism of Lie groupoids.

Finally, the identification between the reduced form and the magnetic form is established in \cite{MR0448428,MR1171218,MR0604285}.
\end{proof}

\begin{remark}\label{rem:affine-poisson}
The connection also splits the central Lie algebroid extension as
$\widetilde A\cong A\oplus\mathbb R$, with
\[
[(X,0),(Y,0)]_{\widetilde A}
=([X,Y]_A,\lambda(X,Y)),
\]
for a Lie algebroid $2$-cocycle $\lambda\in\Gamma(\wedge^2A^*)$.
Write $q:A^*\to M$ for the projection and
$\ell_X(\eta)=\langle\eta,X(q(\eta))\rangle$ for the fibrewise
linear function associated with $X\in\Gamma(A)$. The Poisson
structure on $A^*$ integrated by the magnetic cotangent groupoid is
\[
\begin{aligned}
\{\ell_X,\ell_Y\}&=\ell_{[X,Y]_A}+q^*\lambda(X,Y),\\
\{\ell_X,q^*f\}&=q^*(\rho(X)f),\\
\{q^*f,q^*g\}&=0,
\end{aligned}
\]
for $f,g\in C^\infty(M)$. Indeed, the level-one hyperplane
$N\subset\widetilde A^*$ is a Poisson submanifold of the linear
Poisson manifold $\widetilde A^*$, and its induced Poisson structure
is identified with the one above under $N\cong A^*$.
At moment level $r$, the cocycle term is $r\,q^*\lambda(X,Y)$.
\end{remark}

Two normalized connections give gauge-equivalent magnetic models.
Indeed, if $\theta_1-\theta_2=p^*\alpha$, then
\[
\gamma_1-\gamma_2=\partial\alpha,\qquad
\omega_{B,1}-\omega_{B,2}=d\alpha,\qquad
\beta_1=\beta_2\circ F_\alpha,
\]
where $F_\alpha(\xi_k)=\xi_k+\alpha_k$. Thus Proposition
\ref{prop:gauge-2} identifies the two magnetic cotangent presentations.
The same formulas at moment level $r$ give
$(T^*\mathcal G)_{r\gamma}$ with symplectic form
$\omega_{\mathrm{can}}+r\operatorname{pr}_{\mathcal G}^*\omega_B$.

\subsection{Lie groups and a nontrivial Kac--Moody deformation}\label{sec:Kac-Moody-section}
For a central extension of Lie groups
$1\to\mathbb S^1\to\widetilde G\to G\to1$, choose a linear
splitting $\widetilde{\mathfrak g}\simeq\mathfrak g\oplus\mathbb R$
and the left-invariant connection
$\theta_{\tilde g}((L_{\tilde g})_*(v,r))=r$.
The corresponding affine coadjoint cocycle is
\[
\langle\chi(g),v\rangle
 =\operatorname{pr}_{\mathbb R}
       \bigl(\operatorname{Ad}_{\tilde g^{-1}}(v,0)\bigr),
\qquad p(\tilde g)=g,
\]
which is independent of the lift. In left-invariant coordinates,
\eqref{eq:descend-theta} becomes
\begin{equation}\label{eq:gamma-chi-central}
\gamma^\chi_{(g,h)}\bigl((L_g)_*v,(L_h)_*w\bigr)
=-\langle\chi(h),v\rangle.
\end{equation}
Consequently, under left trivialization, $(T^*G)_{r\gamma^\chi}$
is the affine action groupoid $G\times_{r\chi}\mathfrak g^*$, for
$g\cdot_r\eta=\operatorname{Ad}_{g^{-1}}^*\eta+r\chi(g)$,
equipped with the magnetic forms
$\omega_{\mathrm{can}}+r\operatorname{pr}_G^*\omega_B$.
At $r=0$ it is the ordinary coadjoint action groupoid.

\begin{propn}\label{prop:lie-group-normal-form}
Let $G$ be a Lie group and let $\gamma\in\Omega^1(G\times G)$
satisfy the cocycle and normalization conditions of Theorem A.
Then there exist $\alpha\in\Omega^1(G)$ and a smooth coadjoint
$1$-cocycle $\chi:G\to\mathfrak g^*$ such that
\[
\gamma=\gamma^\chi+\partial\alpha,
\qquad
\chi(gh)=\chi(g)+\operatorname{Ad}_{g^{-1}}^*\chi(h),
\]
where $\gamma^\chi$ is given by \eqref{eq:gamma-chi-central}.
Consequently, $(T^*G)_\gamma$ is gauge isomorphic to the affine
coadjoint action groupoid $G\times_\chi\mathfrak g^*$.
\end{propn}

\begin{proof}
In left trivialization, write
\[
\gamma_{(g,h)}((L_g)_*u,(L_h)_*v)
=\langle P(g,h),u\rangle+\langle Q(g,h),v\rangle.
\]
Comparing the coefficients of the third tangent vector in
$\partial\gamma=0$ gives
$Q(g,h)=Q(gh,e)-Q(h,e)$. Define
$\alpha_g((L_g)_*u)=-\langle Q(g,e),u\rangle$.
Then $\gamma':=\gamma-\partial\alpha$ has no component in the
second tangent variable. Writing its first component as $P'(g,h)$,
the coefficient of the middle tangent vector in
$\partial\gamma'=0$ gives $P'(gh,k)=P'(h,k)$; thus
$P'(g,h)=-\chi(h)$ is independent of $g$. The first tangent
component of the same identity now implies
$\chi(hk)=\chi(h)+\operatorname{Ad}_{h^{-1}}^*\chi(k)$.
Hence $\gamma'=\gamma^\chi$, and Proposition~\ref{prop: gauge-1}
gives the required gauge isomorphism. Under left trivialization,
$(T^*G)_{\gamma^\chi}$ is precisely the affine coadjoint action
groupoid.
\end{proof}

The Kac--Moody extension
$1\to\mathbb S^1\to\widetilde{LG}\to LG\to1$
provides a geometrically nontrivial instance. For the loop-group example, we use the notation and conventions
of \cite{A-M-M}, Section~8.1.
\begin{propn}\label{thm:kac-nonmorita}
Let $G$ be a compact, connected, nontrivial Lie group equipped
with a Kac--Moody extension of $LG$.
Then the coadjoint action groupoid
$LG\times_0L\mathfrak g^*$ and the gauge action groupoid
$LG\times_\chi L\mathfrak g^*$ are not Morita equivalent.
\end{propn}
\begin{proof}
Morita equivalent Lie groupoids have isomorphic isotropy Lie
groups at corresponding objects \cite{stacks_B-X}.
The coadjoint action fixes $0\in L\mathfrak g^*$, whose
isotropy group is $LG$. Thus Morita equivalence would require
a gauge stabilizer isomorphic to $LG$.

By the holonomy description in \cite{A-M-M}, Section~8.1,
each gauge stabilizer is isomorphic to a centralizer in $G$,
and is therefore finite-dimensional. No such stabilizer is
isomorphic to $LG$, proving the claim.
\end{proof}

For each $r$, the $\mathbb{S}^1$- bundle $K_r:=\mu^{-1}(r)\to
\Gamma_r:=K_r/\mathbb S^1$ is a central extension of groupoids.
It therefore presents an $\mathbb S^1$-gerbe over the
corresponding quotient stack.
For $G$ compact, simple, and simply connected, with the basic
Kac--Moody normalization, the level-$1$ construction recovers
the basic gerbe over $[G/G]$, whose Dixmier--Douady class
generates $H^3_G(G;\mathbb Z)$ \cite{A-M-M,B-X-Z}.

\vspace{20pt}  

\noindent \textbf{Acknowledgments.}
We are especially grateful to Ping Xu for suggesting the study of deformed groupoid structures, Luen-Chau Li for clarifying magnetic-form conventions, and Marco Zambon for suggestions on further directions. We also thank Zhuo Chen, Miquel Cueca, Wenda Fang, Žan Grad, Nigel Higson, Noriaki Ikeda, Honglei Lang, Zhangju Liu, Yiannis Loizides, João Nuno Mestre, Xiaomeng Xu, Mathieu Sti\'enon, and Bin Zhang for helpful discussions and feedback.

\end{document}